\documentclass{article}
\usepackage[margin=1in]{geometry}
\usepackage{amsmath,amsfonts,
mathrsfs,xypic,amssymb,hyperref,amsthm}
\usepackage{multicol}
\usepackage{mathtools}
\usepackage[numbers]{natbib}
\hypersetup{
    colorlinks=true,
    linkcolor=blue,
    filecolor=magenta,      
    urlcolor=cyan,
    citecolor=magenta}
\usepackage{tikz}
\makeatletter
\newcommand\mathcircled[1]{%
  \mathpalette\@mathcircled{#1}%
}
\newcommand\@mathcircled[2]{%
  \tikz[baseline=(math.base)] \node[draw,circle,inner sep=1pt] (math) {$\m@th#1#2$};%
}
\makeatother

\newcommand{\ve}{\varepsilon}

\newcommand{\A}{\mathscr A}
\newcommand{\FF}{\mathbb F}
\newcommand{\digr}{\operatorname{dir}}
\newcommand{\xdir}{x^{\digr}}
\newcommand{\out}{\operatorname{out}}
\newcommand{\inn}{\operatorname{in}}

\newtheorem{theorem}{Theorem}[section]
\newtheorem{corollary}[theorem]{Corollary}
\newtheorem{proposition}[theorem]{Proposition}
\newtheorem{lemma}[theorem]{Lemma}
\theoremstyle{definition}
\newtheorem{example}{Example}[section]
\newtheorem{remark}[example]{Remark}

\begin{document}
\author{Donald M.\ 
Larson\footnote{The Catholic
University of America,
620 Michigan Ave NE, Washington
DC 20064
}}
\title{Connectedness
of graphs
arising from the dual Steenrod
algebra}

\maketitle
%%MENTION WHITNEY'S THEOREM ON
%2-CONNECTIVITY (Bondy p. 44) in open
%questions section!  Is there a Steenrod
%interpretation of this
%"connectivity" result?

\begin{abstract}
We establish connectedness
criteria for graphs associated
to monomials in certain quotients
of the mod 2 dual Steenrod algebra
$\A^*$.
We also investigate questions about trees and Hamilton 
cycles in the context of these
graphs.  Finally, we improve
upon a known connection between
the graph theoretic interpretation
of $\A^*$ and its structure
as a Hopf algebra.\\


\noindent{\bf Mathematics
Subject Classification 2020.}
\href{
https://mathscinet.ams.org/mathscinet/msc/msc2020.html?t=55Sxx&btn=Current
}{55S10}, 
\href{
https://mathscinet.ams.org/mathscinet/msc/msc2020.html?t=16Txx&btn=Current
}{16T05}, 
\href{
https://mathscinet.ams.org/mathscinet/msc/msc2020.html?t=05Cxx&btn=Current
}{05C90}
\end{abstract}

\section{Introduction}
\label{Sec:Intro}
The mod
2 Steenrod algebra $\A_*$ and
its dual $\A^*$ act on mod
2 cohomology and homology,
respectively, making them
indispensable computational
tools in homotopy theory
at the prime 2.  It is therefore
desirable to study 
these algebras from a variety
of 
points of view (see, e.g., 
\cite{Smith,Wood3,Wood4,Wood2}
or the bibliography of \cite{Wood}).
The purpose of this paper is to
study the mod 2 dual Steenrod algebra
$\A^*$ from the point of view of
graph theory, as first advocated
by R.\ M.\ W.\ Wood in \cite[\S8]{Wood}
and subsequently advanced by
C.\ Yearwood in \cite{Yearwood}.  

The algebra $\A^*$ has the structure
of a graded polynomial algebra.
More precisely,
$\A^*=\FF_2[\xi_1,\xi_2,\xi_3,
\ldots]$
where $|\xi_i|=2^i-1$.
Given an integer $n\geq0$, we
may form the truncated polynomial
algebra $\A^*(n)=\A^*/I(n)$, where
\begin{equation}\label{Eqn:Ideal}
I(n)=\left(\xi_1^{2^{n+1}},
\xi_2^{2^n},\xi_3^{2^{n-1}},\ldots,
\xi_{n+1}^2,\xi_{n+2},
\xi_{n+3},\ldots
\right).
\end{equation}
Milnor \cite{Milnor} dualized
these finite quotients over
$\FF_2$ to yield
finite subalgebras
$\A_*(n)\subset\A_*$
that facilitate computations of homotopy groups.
For example, Adams spectral
sequence computations  
involving Ext for modules
over $\A_*$ can sometimes
be done instead over $\A_*(n)$
for some $n$,
via a change-of-rings
isomorphism 
\cite[A1.3.12]{Ravenel} (see also
\cite[\S2]{Bailey} for concrete examples of this phenomenon).
%VERIFY THE NUMBERING OF THIS
%REFERENCE IN 1986 VERSION.
%ALSO CITE BAILEY JPAA
%AS AN EXAMPLE!

Wood
gives a graph theoretic interpretation
of the quotients $\A^*(n)$
by associating to every
monomial $x\in\A^*(n)$ a graph,
which by abuse of notation we shall
also denote $x$,
on the vertex set 
$\{2^0,2^1,\ldots,2^{n+1}\}$.
We shall describe the construction
of these graphs in 
Section \ref{Sec:GraphTheory}.
% One of our aims is to 
% characterize connectedness 
% of the graph $x$
% using only the data of its
% underlying monomial.  
Let us write
$x=\xi_1^{r_1}\xi_2^{r_2}
\cdots\xi_{n+1}^{r_{n+1}}$, where
$0\leq r_i\leq 2^{n+2-i}-1$
and where each $r_i$ has the
dyadic expansion
\begin{equation}
\label{Eqn:Dyadic}
r_i=\sum_{m=0}^{n+1-i}
a_{n+1-m,n+1-i-m}2^{n+1-i-m}.
\end{equation}
Our first theorem gives 
connectedness
criteria for $x$ using only the 
data of its underlying monomial.
\begin{theorem}\label{Thm:Main}
%With $x\in\A^*(n)$ as above,
The graph
$x$ is connected if and only if
the integers
\begin{equation}\label{Eqn:Conn}
C(p,q):=\sum_{t=1}^{n+1}\sum_T\prod_{k=1}^t
a_{\min(p_{k-1},p_k),\max(p_{k-1},p_k)}
\end{equation}
are positive for all $0\leq p< q\leq n+1$,
where $T$ is the set of all $(t+1)$-tuples
$(p_0,p_1,\ldots,p_t)\in\{0,\ldots,n+1\}^{t+1}$
such that $p=p_0\neq p_1\neq\cdots\neq p_t=q$.
\end{theorem}
With
Theorem \ref{Thm:Main} in hand, we may
easily identify
which graphs $x$ are trees. Let $\alpha(m)$
denote the number of 1s in the dyadic expansion
of an integer $m$.
\begin{theorem}\label{Thm:Trees}
The graph $x$ is a tree
if and only if the integers
\eqref{Eqn:Conn} are positive (that is
to say, $x$ is connected) and 
$\displaystyle\sum_{i=1}^{n+1}\alpha(r_i)=n+1$.
\end{theorem}

In \cite{Yearwood}, C.\ Yearwood 
sharpens the ideas in \cite{Wood} by
carefully 
interpreting the Hopf algebra structure
on $\A^*$ and its quotients $\A^*(n)$
graph theoretically.  We shall have more to say about this
in Section \ref{Sec:Hopf}.
One important
implication of
Yearwood's work is that the graphs
$x$ are perhaps most naturally viewed
as {\em digraphs}, with edges
oriented in the direction of
the larger vertex
(i.e., $2^{\alpha}\to 2^{\beta}$ if 
$0\leq\alpha<\beta\leq n+1$).
With this in mind, we offer
a digraph version of Theorem \ref{Thm:Main}.
Let $x^{\digr}$ denote the graph $x$ viewed as a digraph.  The
appropriate connectedness property
to study
for $x^{\digr}$ is that of
being {\em unilateral} (see Section \ref{Sec:GraphTheory}).

\begin{theorem}\label{Thm:Dir}
The digraph $x^{\digr}$ is 
unilateral
if and only if the integers
\begin{equation}\label{Eqn:Uni}
U(p,q):=\sum_{t=1}^{n+1}\sum_{T'}\prod_{k=1}^t
a_{p_{k-1},p_k}
\end{equation}
are positive for all $0\leq p< q\leq n+1$, 
where $T'$ is the set of all
$(k+1)$-tuples $(p_0,p_1,\ldots,p_t)
\in\{0,\ldots,n+1\}^{t+1}$
such that $p=p_0<p_1<\cdots<p_t=q$.
\end{theorem}

Among the questions posed by
Wood in \cite{Wood} regarding his 
graph theoretic interpretation
of $\A^*(n)$ is whether there are
algebraic analogs of classical 
questions about Hamilton 
cycles.  In response, we offer
the following result.
\begin{theorem}
\label{Thm:Ham}
The graph $x\in\A^*(n)$ 
has a Hamilton cycle if 
$n>0$ and for every vertex
$2^j$ of $x$, 
\[
\#\{1\leq k\leq j\>:\>
a_{j,j-k}=1\}+
\#\{1\leq k\leq n+1-j\>:\>
a_{j+k,j}=1\}
\geq\frac n2.
\]
Moreover, the digraph 
$x^{\digr}$ has a directed
Hamilton path if and only if
$x$ is divisible by 
$\xi_1^{2^{n+1}-1}$.  
\end{theorem}

The aforementioned Hopf algebra structure
on $\A^*(n)$, to be described in Section
\ref{Sec:Hopf}, includes a coproduct
$\Delta:\A^*(n)\otimes\A^*(n)\to\A^*(n)$
and an antipode
$c:\A^*(n)\to \A^*(n)$. 
Our last theorem is a generalization
of Lemmas 3.1.7 and 3.1.8 of \cite{Yearwood}.
\begin{theorem}
\label{Thm:Hopf} 
Let $\xi_i^{2^j}\in\A^*(n)$
(whose underlying graph is the single
edge connecting $2^j$ and $2^{i+j}$;
see Section \ref{Sec:GraphTheory}).
Then the coproduct $\Delta(\xi_i^{2^j})
\in\A^*(n)$ is the sum of tensors
of all pairs of edges that make length 2
directed paths from $2^j$ to $2^{i+j}$,
and the antipode $c(\xi_i^{2^j})
\in\A^*(n)$ is the sum of all directed
paths from $2^j$ to $2^{i+j}$.
\end{theorem}
As a corollary, we obtain another
characterization of 
the unilaterality of $x^{\digr}$.
\begin{corollary}
\label{Cor:Hopf}
For $x\in\A^*(n)$,
the graph $x^{\digr}$ is unilateral
if and only if for each 
$\xi_i^{2^j}\in\A^*(n)$,
at least one summand of
$c(\xi_i^{2^j})$ is a factor of $x$.
\end{corollary}

The paper is structured as
follows.  Section 
\ref{Sec:GraphTheory}
establishes the necessary 
terminology from graph theory
and describes Wood's construction
of the graphs corresponding to
monomials $x\in\A^*(n)$.  In
Section \ref{Sec:Proofs} we
prove our connectedness
criteria, namely
Theorems \ref{Thm:Main}
and \ref{Thm:Dir}.
Section \ref{Sec:Trees} 
discusses trees and
Hamilton cycles and 
contains the proofs of Theorems
\ref{Thm:Trees} and
\ref{Thm:Ham}.  Finally,
in Section \ref{Sec:Hopf}, 
we discuss the Hopf algebra
structure of $\A^*(n)$ in the
context of graph theory, 
prove Theorem \ref{Thm:Hopf}
and Corollary \ref{Cor:Hopf},
and pose some open questions.

\subsection*{Acknowledgements}
The author would like to thank 
Caroline Yearwood, whose thesis 
was a major impetus for this work,
and her advisor Kyle Ormsby, who
kindly provided a copy of Yearwood's
thesis.  The author would also like to
thank Kiran Bhutani and Paul Kainen for
many helpful comments and suggestions.

\section{Graph theory background and Wood's construction}
\label{Sec:GraphTheory}
The purposes of this section
are to recall the 
necessary definitions from
graph theory and to describe
Wood's construction of graphs
associated to monomials in
$\A^*(n)$.  Our main graph
theory references
are \cite{Bondy} and \cite{West}.

\subsection{Definitions, 
conventions, and notation}
\label{Subsec:GraphDefns}
A {\bf graph} is an ordered
pair $G=(V_G,E_G)$ where
$V_G$ is the set of vertices,
and where $E_G$, the set of edges, is a subset of
the set $\binom{V_G}{2}$ of
unordered pairs of vertices
in $V_G$.  If 
$e=\{v_0,v_1\}\in E_G$,
we say
$v_0$ and $v_1$ are the 
{\bf ends} of $e$. 
All graphs in this paper
are {\bf finite}, meaning
$V_G$ and $E_G$ are finite sets,
and
{\bf simple}, meaning
no edge has
identical ends and no
two edges have the same pair
of ends.
A {\bf walk} in $G$ is a finite
non-empty sequence $v_0,v_1,\ldots,
v_k$ of vertices such that each
consecutive pair $v_i,v_{i+1}$
comprises the ends of an
edge in $E_G$.
If the vertices of a walk in $G$ are
distinct, it is called a 
{\bf path} in $G$.
We say the walk or path $v_0,v_1,\ldots,v_k$
has {\bf length} $k$.
We say $G$ is {\bf connected} if there exists
a path connecting any two distinct vertices.  
If $S$ is a set of vertices,
a {\bf complete graph}
on $S$ is a graph whose
vertex set is $S$ and whose
edge set contains one edge
for every pair of 
distinct vertices in $S$.  

A {\bf cycle} in $G$ is a finite
sequence $v_0,v_1,\ldots,v_k,
v_0$ of vertices such that
$v_0,v_1,\ldots,v_k$ is a path
and the pair $v_k,v_0$ comprises
the ends of an edge in $E_G$.
% A graph $H$ is a {\bf subgraph}
% of $G$ if $V_H\subset V_G$,
% $E_H\subset E_G$, and 
% $\psi_H$ is the restriction
% of $\psi_G$ to $E_H$.
% A subgraph $H$ of $G$ is 
% a {\bf tree}
% if it is maximal among all
% subgraphs of $G$ that do not
% contain a cycle.  
We say $G$ is {\bf acyclic} if it
contains no cycles.  
A {\bf tree} is an acyclic connected
graph.
A
{\bf Hamilton cycle} in $G$
is a cycle containing every
vertex of $G$.  The {\bf degree}
$\deg(v)$ of a vertex $v\in V_G$
is the number of edges 
in $E_G$ that have $v$ as an end.

A {\bf directed graph} (or
{\bf digraph}) is an
ordered pair $D=(V_D,E_D)$ where $V_D$ is the 
set of vertices, and where
$E_D$, the set of edges, 
is a subset of
the set of
{\em ordered} pairs of vertices in $V_G$.
If $e=(v_0,v_1)\in E_D$, 
we say $v_0$ is the
{\bf tail} of $e$ and $v_1$
is the {\bf head} of $e$.
Any digraph
has an underlying graph
by replacing each
ordered pair 
$(v_0,v_1)\in E_D$
with the corresponding
unordered pair
$\{v_0,v_1\}$.
All directed
graphs in this paper
have underlying graphs
that are 
finite and simple. A {\bf directed
walk} in $D$ is a finite 
non-empty sequence
$v_0,v_1,\ldots,v_k$ of 
vertices such that
each for consecutive pair
$v_i,v_{i+1}$, there is
an edge in $E_D$ with
tail $v_i$ and head $v_{i+1}$.
If the vertices of a directed
walk in $D$ are distinct, it is called
a {\bf directed path} in $D$.
As with walks or paths, 
we say the directed walk
or directed path
$v_0,v_1,\ldots,v_k$ has
{\bf length} $k$.  
A digraph $D$ is {\bf unilateral}
(see \cite[Exercise 10.2.2]{Bondy})
if for every pair of 
distinct vertices
$v_i,v_j\in V_D$, there is
a directed path starting at
$v_i$ and ending at $v_j$ or
vice versa.
A {\bf Hamilton directed path}
in $D$ is a directed path
containing every vertex of $D$.  
The {\bf out-degree} 
$\deg_{\out}(v)$
of a vertex
$v\in V_D$ is the number of edges
in $D$ that have $v$ as a tail,
and the {\bf in-degree}
$\deg_{\inn}(v)$ of $v$ is the number
of edges in $D$ that have $v$ as a
head.
Note that $\deg(v)=\deg_{\out}(v)
+\deg_{\inn}(v)$, where
we interpret
$\deg(v)$ as being the
degree of $v$ in the graph
underlying the digraph $D$.

If $G$ is a graph with ordered
vertex set $V_G=\{v_0,\ldots,
v_{n-1}\}$,
the {\bf adjacency matrix}
of $G$ is the $n\times n$
matrix $A_G$ with $(p,q)$th entry
equal to the number of edges
with ends $v_p$ and $v_q$.
(Here, we use $p$ and $q$ 
as row and column indices, respectively,
where $0\leq p\leq n-1$ and
$0\leq q\leq n-1$.)
If $D$ is a digraph with
vertex set $V_D=\{v_0,\ldots,
v_{n-1}\}$, the
{\bf adjacency matrix} of $D$
is the $n\times n$ matrix
$A_D$ with $(p,q)$th entry
equal to the number of edges
with tail $v_p$ and head $v_q$.
As a result of our conventions, 
all adjacency matrices $A_G$
associated to graphs $G$
in this paper
will be symmetric, binary (that is,
all entries either 0 or 1), and 
will have zeros along the main diagonal. Similarly, all adjacency matrices
$A_D$ associated to digraphs $D$ in this
paper will be binary and strictly
upper triangular.

\subsection{Wood's construction}
\label{Subsec:WoodConstr}
We now explain how Wood encodes
the algebras $\A^*(n)$ in terms
of graphs.  Recall from Section
\ref{Sec:Intro} that the graph $x$
underlying a monomial $x\in\A^*(n)$
will have vertex set $V_x=\{
2^0,2^1,\ldots,2^{n+1}\}$.

To begin, consider a monomial 
of the form $\xi_i^{2^j}\in\A^*(n)$
where $1\leq i\leq n+1$ and
$1\leq j\leq 2^{n+2-i}-1$.
Wood's construction declares that
the graph $\xi_i^{2^j}$
has edge set consisting of
a single edge with ends $2^j$ and
$2^{i+j}$.
\begin{example}\label{Ex:Edges}
%\begin{multicols}{2}
\begin{enumerate}
\item[(a)] $\xi_2\in\A^*(2)$
\[
\entrymodifiers={++[o][F-]}
\xymatrix%@-1pc
{
8&
1\ar@{-}[dl]\\
4&2
}\]
\item[(b)] $\xi_3^2\in\A^*(3)$
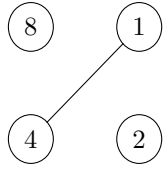
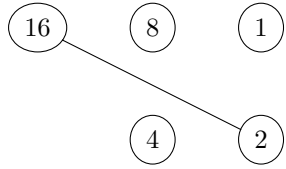
\[
\entrymodifiers={++[o][F-]}
\xymatrix{
16\ar@{-}[rrd]&8&1\\
*{}&4&2
}\]
\end{enumerate}
%\end{multicols}
\end{example}
\noindent If $x=\xi_1^{r_1}
\xi_2^{r_2}\cdots\xi_{n+1}^{r_{n+1}}
\in\A^*(n)$,
then we may use the dyadic expansions given in \eqref{Eqn:Dyadic} to
write $x$ as a product of monomials
of the form $\xi_i^{2^j}$,
and Wood's construction declares
the graph $x$ to have edge set
equal to 
the disjoint union of the corresponding edges.  
\begin{example}\label{Ex:MultEdges}
%\begin{multicols}{2}
\begin{enumerate}
\item[(a)]
$
x=\xi_1^6\xi_2\xi_3=
\xi_1^2\xi_1^4\xi_2\xi_3\in
\A^*(2)$
\[
\entrymodifiers={++[o][F-]}
\xymatrix%@-1pc
{
8&1\ar@{-}[dl]\ar@{-}[l]\\
4\ar@{-}[u]&2\ar@{-}[l]
}\]
\item[(b)] $x=\xi_1^{15}
\xi_3^2
=\xi_1\xi_1^2\xi_1^4\xi_1^8
\xi_3^2
\in\A^*(3)$
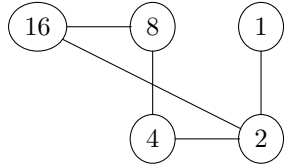
\[
\entrymodifiers={++[o][F-]}
\xymatrix{
16\ar@{-}[rrd]&8\ar@{-}[l]
&1\ar@{-}[d]%\ar@{-}[l]
\\
*{}&4\ar@{-}[u]&2\ar@{-}[l]
}\]
\end{enumerate}
%\end{multicols}
\end{example}
Any graph on the vertices
$\{2^0,2^1,\ldots,2^{n+1}\}$
has a corresponding monomial
in $\A^*(n)$.
Following the notation
in \cite{Wood},
we shall denote the top 
degree class
$\xi_1^{2^{n+1}-1}
\xi_2^{2^{n}-1}\cdots
\xi_{n+1}\in\A^*(n)$
by $\Delta$. The corresponding
graph $\Delta$ is a complete
graph on $\{2^0,2^1,\ldots,
2^{n+1}\}$.  At the other
extreme, the graph corresponding
to $1\in\A^*(n)$ is the
graph on $\{2^0,2^1,\ldots,
2^{n+1}\}$ with empty edge
set.  
% We also note that 
% $0\in\A^*(n)$ corresponds to any
% graph that attempts to include an
% ``out of bounds'' vertex (or, in
% algebra terms, any graph
% coming from a monomial in 
% the coset $0+I(n)$ of $I(n)$
% in $\A^*$), as shown below in 
% Example \ref{Ex:Extremes}(c).
\begin{example}\label{Ex:Extremes}
%\begin{multicols}{2}
\begin{enumerate}
\item[(a)]
$\Delta=
\xi_1^{15}\xi_2^7\xi_3^3
\xi_4
\in\A^*(3)$
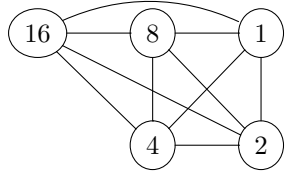
\[
\entrymodifiers={++[o][F-]}
\xymatrix{
16\ar@{-}[rrd]
\ar@{-}[rd]
\ar@/^1pc/@{-}[rr]
&8\ar@{-}[l]
\ar@{-}[rd]
&1\ar@{-}[d]\ar@{-}[l]
\\
*{}&4\ar@{-}[u]
\ar@{-}[ur]
&2\ar@{-}[l]
}
\]
\item[(b)]
$1\in \A^*(3)$
\[
\entrymodifiers={++[o][F-]}
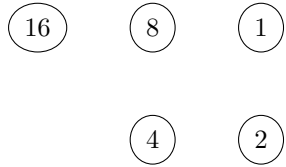
\xymatrix{
16&8&1\\
*{}&4&2
}
\]
% \item[(c)] $\xi_1
% \xi_2\xi_3\xi_4^2=0\in\A^*(3)$
% (the vertex $2^5=32$ is ``out of bounds'')
% \[
% \entrymodifiers={++[o][F-]}
% \xymatrix{
% 16&8&1\ar@{-}[d]
% \ar@{-}[dl]\ar@{-}[l]
% \\
% \mathit{32}\ar@/_1pc/@{--}[rr]&4&2
% }
% \]
\end{enumerate}
%\end{multicols}
\end{example}
\begin{remark}
If $x$ and $y$ are arbitrary monomials
in $\A^*(n)$, we can 
obtain the graph $xy$ from the 
individual graphs $x$ and $y$ via
the following procedure: 
(1) Overlay the graphs $x$ and $y$
on their common vertex set
$\{2^0,2^1,\ldots,2^{n+1}\}$; (2) 
For all pairs of edges between two
vertices $2^i$ and $2^j$, delete the
pair and, if possible, perform a ``carry'' by inserting
an edge between $2^{i+1}$ and 
$2^{j+1}$;
(3) Repeat until no further 
deletions/carries are required.
If $x=\xi_1^{r_1}\xi_2^{r_2}\cdots
\xi_{n+1}^{r_{n+1}}$ and
$y=\xi_1^{s_1}\xi_2^{s_2}\cdots
\xi_{n+1}^{s_{n+1}}$,
the carries of edges one would
perform correspond precisely
to the carries required when adding
the dyadic expansions of $r_i$
and $s_i$ for $1\leq i\leq n+1$.  
In fact, it is shown in
\cite{Yearwood} that if
the set of all possible graphs 
on $\{2^0,2^1,\ldots,2^{n+1}\}$
is endowed with this multiplication, and with 
addition given
by finite formal sums over $\FF_2$, 
the result is an $\FF_2$-algebra 
(referred to as a truncated 
``graph algebra''
in \cite{Yearwood})
that is isomorphic to $\A^*(n)$.  
\end{remark}

We noted in Section \ref{Sec:Intro} 
that it is natural to view 
graphs $x\in\A^*(n)$ as digraphs 
$x^{\digr}$
with edges oriented in the direction
of the larger vertex.  To do this
pictorially, one can put an arrow on each
edge pointing toward the head.  
Here is the graph 
$\xi_1^6\xi_2\xi_3\in\A^*(2)$ from
Example \ref{Ex:MultEdges}(a)
viewed as a digraph 
$(\xi_1^6\xi_2\xi_3)^{\digr}$:
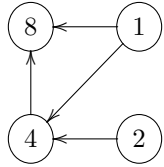
\[
\entrymodifiers={++[o][F-]}
\xymatrix%@-1pc
{
8&1\ar[dl]\ar[l]\\
4\ar[u]&2\ar[l]
}\]
\begin{remark}
In the graph theory literature,
a directed graph $D$ is said
to be {\bf strongly connected} or
{\bf diconnected} (\cite[\S10.1]{Bondy})
if for any ordered pair of vertices
$(v_0,v_1)$, there is a 
directed path starting at $v_0$
and ending at $v_1$ in $D$. 
The way we orient edges in
graphs $x\in\A^*(n)$ makes it
impossible for any
corresponding $\xdir$
to be strongly connected.  
This is why we instead opt for
characterizing
the notion
of $\xdir$ being unilateral
as defined in the
exercises of \cite{Bondy}. 

\end{remark}

\section{Connectedness criteria}\label{Sec:Proofs}

In this section, we prove the 
connectedness criteria in
Theorems \ref{Thm:Main} and
\ref{Thm:Dir} using the basic
theory of adjacency matrices.
\subsection{Connectedness
criterion for $x\in\A^*(n)$}
\label{Subsec:MainProof}
We now
prove Theorem
\ref{Thm:Main}, which
asserts that the graph
$x\in\A^*(n)$ is connected if and only if
the integers $C(p,q)$ defined
in \eqref{Eqn:Conn} are 
positive.  

Let $G$ be a graph 
with ordered vertex set $V_G=
\{v_0,v_1,\ldots,v_{n-1}\}$
and corresponding adjacency matrix $A_G$.  
Given a positive integer $t$,
the $(p,q)$th entry of the
matrix $(A_G)^t$
is equal to the number of 
distinct walks of length $t$ in $G$ between
$v_p$ and $v_q$.  We have
the following well-known result from graph
theory as a consequence.  
\begin{proposition}\label{Prop:Conn}
The graph $G$ is connected
if and only if for all $0\leq p<q\leq n-1$, the $(p,q)$th entry of the matrix
\[
A_G+(A_G)^2+\cdots+(A_G)^{n-1}
\]
is positive.
\end{proposition}

Let $x=\xi_1^{r_1}\xi_2^{r_2}
\cdots\xi_{n+1}^{r_{n+1}}\in\A^*(n)$
as in Section \ref{Sec:Intro}.  
Recall that the vertex
set of the graph $x$ is $V_x=\{2^0,2^1,\ldots,
2^{n+1}\}$.
Following the notation and conventions
from Subsection 
\ref{Subsec:GraphDefns}, let
$A_x$ denote the
$(n+2)\times(n+2)$ adjacency
matrix of $x$, whose rows and 
columns we shall index by
$p$ and $q$, respectively,
with $0\leq p\leq n+1$ and
$0\leq q\leq n+1$.  The key to our proofs
of Theorems \ref{Thm:Main} and 
\ref{Thm:Dir}
is the
connection between the dyadic expansions
of the exponents $r_i$ given in \eqref{Eqn:Dyadic}
and the entries of $A_x$.

\begin{lemma}[\cite{Wood},
or Lemma 3.2.4 of \cite{Yearwood}]
\label{Lem:Entry}
For $0\leq p<q\leq n+1$, the
$(p,q)$th entry of $A_x$ is
$a_{p,q}$, where $a_{p,q}$
is the coefficient on $2^q$ in the
dyadic expansion of $r_{p-q}$,
as in \eqref{Eqn:Dyadic}.
\end{lemma}
\begin{proof}
It follows from Wood's construction
(Subsection \ref{Subsec:WoodConstr})
and the observation preceding
Proposition \ref{Prop:Conn} that 
the $i$th superdiagonal of $A_x$, 
read from bottom to top, yields
precisely the coefficients of 
the dyadic expansion of $r_i$
in order from the largest power of 2
to the smallest.
\end{proof}
Let $t$ be a fixed positive integer
such that $t\leq n+1$.  Since $A_x$ is
symmetric, it follows from Lemma
\ref{Lem:Entry} and the definition
of matrix multiplication that the integer
\[
\sum_T\prod_{k=1}^t
a_{\min(p_{k-1},p_k),\max(p_{k-1},p_k)},
\]
where $T$ is the set of all $(t+1)$-tuples
$(p_0,p_1,\ldots,p_t)\in\{0,\ldots,n+1\}^{t+1}$ such that $p=p_0\neq p_1\neq\cdots\neq
p_t=q$, is precisely the $(p,q)$th entry
of $(A_x)^t$.  Therefore, Proposition
\ref{Prop:Conn} implies that $x$ is
connected if and only if the integers
$C(p,q)$ are positive for all 
$0\leq p<q\leq n+1$.  This proves
Theorem \ref{Thm:Main}.

\subsection{Connectedness
criterion for $x^{\digr}$}

We now prove Theorem \ref{Thm:Dir}, the
digraph analog of Theorem \ref{Thm:Main},
which asserts that the digraph $\xdir$ is
unilateral if and only if the integers
$U(p,q)$ defined in \eqref{Eqn:Uni} are
positive. The proof will be similar
in format to the proof 
of Theorem \ref{Thm:Main} given in
Subsection \ref{Subsec:MainProof} but
with adjustments made to accommodate
the directed case as needed.  

Let $D$ be a digraph with ordered vertex set
$V_D=\{v_0,v_1,\ldots,v_{n-1}\}$ and
corresponding 
adjacency matrix $A_D$.  Given a 
positive integer $t$, the $(p,q)$th entry
of the matrix $(A_D)^t$ is equal to the 
number of distinct directed walks of 
length $t$ in $D$ starting at $v_p$ and
ending at $v_q$.  We therefore have the
following digraph analog of 
Proposition \ref{Prop:Conn}.
\begin{proposition}\label{Prop:Dir}
The digraph $D$ is unilateral if and only if
for all $0\leq p\neq q\leq n-1$, either
the
$(p,q)$th entry or the $(q,p)$th entry
of the matrix
\[
A_D+(A_D)^2+\cdots+(A_D)^{n-1}
\]
is positive.
\end{proposition}
Once again, let $x=\xi_1^{r_1}\xi_2^{r_2}
\cdots\xi_{n+1}^{r_{n+1}}\in\A^*(n)$
as in Section \ref{Sec:Intro}.
Let $\xdir$ be the underlying digraph, 
and let $A_{\xdir}$ denote its
$(n+2)\times(n+2)$ adjacency matrix.  
Because edges in
$\xdir$ are always oriented toward the
larger vertex, the $(p,q)$th entry
of $A_{\xdir}$ is equal to the
$(p,q)$th entry of $A_x$ if $p<q$, and
is equal to zero otherwise.  This implies
that Lemma \ref{Lem:Entry} holds with
$A_{\xdir}$ in place of $A_x$.
For a fixed positive integer $t\leq n+1$,
it follows from this altered version
of Lemma \ref{Lem:Entry} and the 
definition of matrix multiplication for
strictly upper triangular matrices
that the integer 
\[
\sum_{T'}\prod_{k=1}^t a_{p_{k-1},p_k},
\]
where $T'$ is the set of all $(k+1)$-tuples
$(p_0,p_1,\ldots,p_t)\in\{0,\ldots,n+1\}^{t+1}$ such that $p=p_0<p_1<\cdots<p_t=q$, 
is precisely the $(p,q)$th entry of
$(A_{\xdir})^t$.  Therefore, Proposition
\ref{Prop:Dir} implies $\xdir$ is 
unilateral if and only if
the integers $U(p,q)$ are positive
for all $0\leq p<q\leq n+1$.  This proves
Theorem \ref{Thm:Dir}.

\subsection{Examples}
We present two examples
of Theorems \ref{Thm:Main} and \ref{Thm:Dir}
applied to graphs in $\A^*(n)$.  
Consider first
$\xi_1^6\xi_2\xi_3\in\A^*(2)$ from Example
\ref{Ex:MultEdges}(a).  The integers
$C(p,q)$ associated to the underlying graph
are
% If we denote its
% adjacency matrix by $A$, then
% \[
% A=\begin{pmatrix}
% 0&0&1&1\\
% 0&0&1&0\\
% 1&1&0&1\\
% 1&0&1&0
% \end{pmatrix}
% \]
% and
\begin{equation}\label{Eqn:MatrixOne}
C(0,1)=2,\quad C(0,2)=6,\quad
C(0,3)=5,\quad C(1,2)=4,\quad
C(1,3)=2,\quad C(2,3)=6,
\end{equation}
while among the associated integers $U(p,q)$,
one finds
\begin{equation}\label{Eqn:MatrixTwo}
U(0,1)=0.
\end{equation}
% If we denote the adjacency matrix of 
% $(\xi_1^6\xi_2\xi_3)^{\digr}$ by $A'$, then
% \[
% A'=\begin{pmatrix}
% 0&0&1&1\\
% 0&0&1&0\\
% 0&0&0&1\\
% 0&0&0&0
% \end{pmatrix}
% \]
% and 
Since the integers in \eqref{Eqn:MatrixOne}
are all positive, Theorem \ref{Thm:Main}
implies $\xi_1^6\xi_2\xi_3$ is connected.
On the other hand, \eqref{Eqn:MatrixTwo}
implies
$\xi_1^6\xi_2\xi_3$ is {\em not} unilateral by
Theorem \ref{Thm:Dir}.

Next, consider $\xi_1^{15}\xi_3^2\in\A^*(3)$
from Example \ref{Ex:MultEdges}(b), for which
the associated integers $C(p,q)$ are
\begin{align}\label{Eqn:MatrixThree}
\begin{split}
C(0,1)&=4,\quad C(0,2)=6,\quad C(0,3)=2,\quad
C(0,4)=6,\quad C(1,2)=6,\\
C(1,3)&=12,\quad C(1,4)=6,\quad C(2,3)=5,\quad
C(2,4)=11,\quad C(3,4)=5.
\end{split}
\end{align}
and the associated integers $U(p,q)$ are
\begin{align}\label{Eqn:MatrixFour}
\begin{split}
U(0,1)&=1,\quad U(0,2)=1,\quad U(0,3)=1,\quad
U(0,4)=2,\quad U(1,2)=1,\\
U(1,3)&=1,\quad U(1,4)=2,\quad U(2,3)=1,\quad
U(2,4)=1,\quad U(3,4)=1.
\end{split}
\end{align}
Since the integers in
\eqref{Eqn:MatrixThree} are all positive,
Theorem \ref{Thm:Main} implies 
$\xi_1^{15}\xi_3^2$ is connected.
Since the integers in \eqref{Eqn:MatrixFour}
are all positive, Theorem \ref{Thm:Dir}
implies $\xi_1^{15}\xi_3^2$ is also 
unilateral.

Note that whether a monomial $x$ is
connected and/or unilateral depends
in part on which truncated polynomial
algebra $\A^*(n)$ it lives in.  For example, 
we saw above that $\xi_1^{15}\xi_3^2\in\A^*(3)$
is both connected and unilateral, but if we were
to regard $\xi_1^{15}\xi_3^2$ as an element of
$\A^*(n)$ for $n\geq4$, the monomial
would have neither property
because the vertex $2^5=32$ would not be the end
of any edge.  

\section{Trees and Hamilton
cycles}
\label{Sec:Trees}

The purpose of this section is to
characterize trees and Hamilton cycles
among the graphs $x\in\A^*(n)$ by proving
Theorems \ref{Thm:Trees} and
\ref{Thm:Ham}. 
\subsection{Trees in $\A^*(n)$}
In this subsection we shall prove
the criteria for trees in
Theorem \ref{Thm:Trees}, which 
asserts that $x\in\A^*(n)$ is
a tree (that is, connected and
acyclic) if and only if
the associated integers $C(p,q)$
from \eqref{Eqn:Conn} are positive
and $\displaystyle\sum_{i=1}^{n+1}
\alpha(r_i)=n+1$.
Our starting point is 
a known characterization of trees
given by the following result.
\begin{proposition}
[Theorem 2.1.4(B) of \cite{West}]
\label{Prop:Tree}
A connected graph $G$ with $n$ 
vertices is 
a tree if and only if it has
$n-1$ edges.
\end{proposition}
Using Proposition \ref{Prop:Tree},
Yearwood gives a criterion for when
a {\em connected} graph $x\in\A^*(n)$
is a tree.
\begin{proposition}
[Proposition 3.3.8 of \cite{Yearwood}]
\label{Prop:Yearwood}
A connected graph $x=\xi_1^{r_1}\xi_2^{r_2}
\cdots\xi_{n+1}^{r_{n+1}}\in\A^*(n)$ is a tree
if and only if $\displaystyle\sum_{i=1}^{n+1}
\alpha(r_i)=n+1$.
\end{proposition}
\begin{proof}
It follows from Wood's construction
(Subsection \ref{Subsec:WoodConstr})
that the number of edges of $x$
is the total number of 1s in the 
dyadic expansions of $r_1,\ldots,
r_{n+1}$.  Yearwood's criterion
then follows from Proposition
\ref{Prop:Tree}.
\end{proof}
Theorem \ref{Thm:Trees} follows
immediately 
from Proposition \ref{Prop:Yearwood}
and Theorem \ref{Thm:Main}.
\subsection{Hamilton cycles}
\label{Subsec:Ham}
We now prove Theorem \ref{Thm:Ham}, 
starting with the sufficient condition
for $x\in\A^*(n)$ to have a Hamilton
cycle.  Recall from Section \ref{Sec:Intro} that we must show
$x$ has a Hamilton cycle if $n>0$
and for every vertex $2^j$ of $x$,
\[
\#\{1\leq k\leq j\>:\>
a_{j,j-k}=1\}+
\#\{1\leq k\leq n+1-j\>:\>
a_{j+k,j}=1\}
\geq\frac n2
\]
where $x=\xi_1^{r_1}\xi_2^{r_2}
\cdots\xi_{n+1}^{r_{n+1}}$
and the integers $a_{p,q}$ are the dyadic
coefficients of the exponents
$r_1,\ldots,r_{n+1}$ as described in
\eqref{Eqn:Dyadic}.  
% Let us
% therefore assume,
% for the time being, that $n>0$. 
We begin by
establishing a lemma that counts
the computes the degree of
a vertex in $x$.  
\begin{lemma}\label{Lem:Deg}
The out-degree of a vertex
$2^j$ in $\xdir$ is
\[
\deg_{\out}(2^j)=\#\{
1\leq k\leq n+1-j\>:\>a_{j+k,j}=1\}
\]
and its in-degree is
\[
\deg_{\inn}(2^j)=\#\{
1\leq k\leq j\>:\>a_{j,j-k}=1
\}
\]
so that the degree of
the vertex $2^j$ in $x$ is
\[
\deg(2^j)=\#\{
1\leq k\leq n+1-j\>:\>a_{j+k,j}=1\}
+
\#\{
1\leq k\leq j\>:\>a_{j,j-k}=1
\}.
\]
\end{lemma}
\begin{proof}
Using the dyadic expansions
given in \eqref{Eqn:Dyadic}, we may
factor $x$ as
\begin{equation}\label{Eqn:Xfactor}
x=\prod_{i=1}^{n+1}\prod_{m=0}
^{n+1-i}\xi_i^{a_{n+1-m,n+1-i-m}
2^{n+1-i-m}}.
\end{equation}
By Wood's construction
(Subsection \ref{Subsec:WoodConstr}),
the edges of the digraph $\xdir$ 
with tail $2^j$ correspond 
precisely to 
factors in the product 
\eqref{Eqn:Xfactor} of the form
$\xi_k^{2^j}$ for $1\leq k\leq
n+1-j$ (such an edge has tail
$2^j$ and head $2^{j+k}$).
The factor $\xi_k^{2^j}$ appears
in \eqref{Eqn:Xfactor} 
if and only if the corresponding
dyadic 
coefficient $a_{j+k,j}$ is equal
to 1.  Similarly, Wood's
construction implies the edges
of $\xdir$ with head $2^j$
correspond precisely to factors
in the product \eqref{Eqn:Xfactor}
of the form $\xi_k^{2^{j-k}}$ 
for $1\leq k\leq j$
(such
an edge has tail $2^{j-k}$ and head
$2^j)$.  The factor 
$\xi_k^{2^{j-k}}$ appears in
\eqref{Eqn:Xfactor} if and only if
the corresponding dyadic coefficient
$a_{j,j-k}$ is equal to 1.
This yields the first two equations
of the lemma, and the third follows
from the fact that 
$\deg(v)=\deg_{\out}(v)+
\deg_{\inn}(v)$ for any vertex
$v$, as we noted in Subsection
\ref{Subsec:GraphDefns}.
\end{proof}
Dirac's Theorem from graph
theory gives a sufficient condition
for a graph $G$ to have a Hamilton
cycle in terms of the degrees
of the vertices of $G$.
\begin{theorem}[Dirac's Theorem]
\label{Thm:Dirac}
A graph $G$ with at least 3
vertices has a Hamilton cycle
if $\deg(v)\geq n/2$ for all
vertices $v$ of $G$.
\end{theorem}
% Since we assume $n>0$, the
% graph $x$ has at least 3 vertices.
Lemma
\ref{Lem:Deg} and Theorem 
\ref{Thm:Dirac} together
imply that a sufficient
condition for $x\in\A^*(n)$ 
to have a
Hamilton cycle is to assume that
$n>0$ (so that $x$ has at least
3 vertices) and that for all
vertices $2^j$ of $x$,
\begin{align*}
\frac n2\leq\deg(2^j)&=
\deg_{\out}(2^j)+\deg_{\inn}(2^j)\\
&=\#\{
1\leq k\leq n+1-j\>:\>a_{j+k,j}=1\}
+
\#\{
1\leq k\leq j\>:\>a_{j,j-k}=1
\}.
\end{align*}
This completes the
proof of the first statement
of Theorem \ref{Thm:Ham}.

We now prove the second
statement of Theorem \ref{Thm:Ham},
which asserts $\xdir$ has a directed
Hamilton path if and only if
$x$ is divisible by
$\xi_1^{2^{n+1}-1}$.
The monomial $\xi_i^{2^{n+1}-1}$
is the largest nonzero power
of $\xi_1$ in the truncated
polynomial algebra $\A^*(n)$.
Therefore, 
if $x$ is divisible
by $\xi_i^{2^{n+1}-1}$, we
may write $x$ as
\begin{align*}
x&=\xi_1^{2^{n+1}-1}
\xi_2^{r_2}\xi_3^{r_3}\cdots
\xi_{n+1}^{r_{n+1}}\\
&=\xi_1^{1+2+4+\cdots+2^n}
\xi_2^{r_2}\xi_3^{r_3}\cdots
\xi_{n+1}^{r_{n+1}}\\
&=\xi_1\xi_1^2\xi_1^4\cdots
\xi_1^{2^n}\xi_2^{r_2}
\xi_3^{r_3}\cdots
\xi_{n+1}^{r_{n+1}}.
\end{align*}
The 
first $n+1$ factors 
(i.e., the powers of $\xi_1$)
correspond
precisely to the edges 
in the directed path
\begin{equation}
\label{Eqn:Ham}
2^0\to2^1\to2^2\to\cdots\to2^n
\to2^{n+1}
\end{equation}
which is a Hamilton
directed path we have just
shown is contained in $\xdir$.
Conversely, suppose 
$\xdir$ contains a Hamilton
directed path.  This Hamilton
directed path must contain
all the vertices of $\xdir$
by definition,
in particular $2^0$.  Since
no edge in $\xdir$ has
$2^0$ as its head, the Hamilton
directed path must in fact start
at $2^0$.  If the first edge
of the path went from $2^0$
to $2^j$ for $j>1$, the vertex
$2^1$ could not be contained
in the path, which implies
the first edge must be the
edge from $2^0$ to $2^1$, i.e., 
$\xi_1$.  An analogous argument
shows the next edge in the path
must be the edge $2^1$ to $2^2$,
i.e., $\xi_1^2$.  Continuing
in this manner shows that the
Hamilton directed path in
$\xdir$ must in fact be the 
directed path \eqref{Eqn:Ham},
and that $x$ is divisible by
\[
\xi_1\xi_1^2\xi_1^4\cdots
\xi_1^{2^n}=\xi_1^{2^{n+1}-1}.
\]
This completes the proof
of Theorem \ref{Thm:Ham}.

\subsection{Examples}
We present examples of
Theorems \ref{Thm:Trees}
and \ref{Thm:Ham} applied
to graphs $x\in\A^*(n)$.
Consider first the monomial
$x=\xi_1\xi_2\xi_3\in\A^*(2)$.
For this choice of $x$,
$r_1=r_2=r_3=1$, which
implies 
\[
\sum_{i=1}^3\alpha(r_i)=3
=2+1
\]
and so 
$\xi_1\xi_2\xi_3$ is 
a tree by Theorem
\ref{Thm:Trees}.  A picture
of the graph
$\xi_1\xi_2\xi_3$ is given by
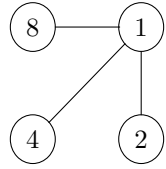
\[
\entrymodifiers={++[o][F-]}
\xymatrix%@-1pc
{
8&1\ar@{-}[dl]
\ar@{-}[d]
\ar@{-}[l]\\
4&2
}\]
and verifies that $\xi_1\xi_2\xi_3$
is indeed a tree.
On the other hand, if we
take $x$ to be $\xi_1^6\xi_2\xi_3
\in\A^*(2)$ from Example
\ref{Ex:MultEdges}(a),
then $r_1=6=4+2$ and
$r_2=r_3=1$, which implies
\[
\sum_{i=1}^3\alpha(r_i)
=4\neq2+1.
\]
We conclude the graph 
$\xi_1^6\xi_2\xi_3$ is
{\em not} a tree by Theorem
\ref{Thm:Trees}.

Next, let us take $x=
\xi_1^6\xi_2^6\xi_3\xi_4
\in\A^*(3)$.
The dyadic coefficients $a_{p,q}$
for 
the exponents $r_1=r_2=6=4+2$
and $r_3=r_4=1$
are given by
\begin{align*}
a_{4,3}&=0,\quad a_{3,2}=1,\quad
a_{2,1}=1,\quad a_{1,0}=0,\\
a_{4,2}&=1,\quad a_{3,1}=1,\quad
a_{2,0}=0,\\
a_{4,1}&=0,\quad a_{3,0}=1,\\
a_{4,0}&=1.
\end{align*}
Using these values, 
we can verify directly that
$\xi_1^6\xi_2^6\xi_3\xi_4$ satisfies the
sufficient condition for having
a Hamilton cycle given in
Theorem \ref{Thm:Ham} by checking
the requisite inequality
at each vertex $v$:
\begin{align*}
v&=1=2^0:\quad
\#\{1\leq k\leq0\>:\>
a_{0,0-k}=1\}+\#
\{1\leq k\leq4\>:\>
a_{0+k,0}=1\}=0+2\geq3/2,
\\
v&=2=2^1:\quad\#\{1\leq k\leq1
\>:\>a_{1,1-k}=1\}+\#\{
1\leq k\leq3\>:\>a_{1+k,1}=1\}
=0+2\geq3/2,\\
v&=4=2^2:\quad\#\{1\leq k\leq2\>:\>
a_{2,2-k}=1\}+\#\{1\leq k\leq2
\>:\>a_{2+k,2}=1\}=1+2\geq
3/2,\\
v&=8=2^3:\quad\#\{1\leq k\leq 3\>:\>
a_{3,3-k}=1\}+\#\{1\leq k\leq1
\>:\>a_{3+k,3}=1\}=3+0\geq
3/2,\\
v&=16=2^4:\quad\#\{1\leq k\leq4\>:\>
a_{4,4-k}=1\}+\#\{1\leq k\leq0
\>:\>a_{4+k,4}=1\}=2+0\geq3/2.
\end{align*}
A picture of the graph $\xi_1^6\xi_2^6
\xi_3\xi_4$ is given by
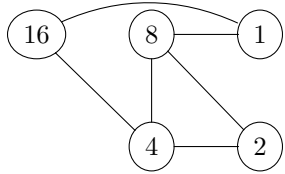
\[
\entrymodifiers={++[o][F-]}
\xymatrix{
16\ar@{-}[rd]\ar@/^1pc/@{-}[rr]&8
\ar@{-}[d]
&1\ar@{-}[l]%\ar@{-}[l]
\\
*{}&4&2\ar@{-}[l]
\ar@{-}[ul]
}\]
and corroborates the presence
of a Hamilton cycle, e.g., 
$2^1,2^2,2^4,2^0,2^3,2^1$.
% If we take $x$ to be 
% $\xi_1^{15}\xi_3^2\in\A^*(3)$ from
% Example \ref{Ex:MultEdges}(b),
% we find that at the vertex
% $2^0=1$, 
% \[
% \#\{1\leq k\leq0\>:\>
% a_{0,0-k}=1\}+\#\{1\leq k\leq4\>:\>a_{0,0+k}=1\}=0+1<3/2
% \]
% so we know $\xi_1^{15}\xi_3^2$

Finally, consider $x=\xi_1^{15}
\xi_3^2\in\A^*(3)$ from Example \ref{Ex:MultEdges}(b).  One can
check that this choice of $x$
does not satisfy the sufficient
condition in Theorem 
\ref{Thm:Ham} for containing a Hamilton
cycle (and as it turns out, it
does not contain a Hamilton
cycle).  However, Theorem \ref{Thm:Ham}
does guarantee that $\xi_1^{15}\xi_3^2$ contains
a Hamilton directed path
since $\xi_1^{15}\xi_3^2$
is divisible by 
$\xi_1^{15}=\xi_1^{2^{3+1}-1}$.
The Hamilton directed path
it contains is
\[
2^0\to 2^1\to 2^2\to 2^3\to 2^4
\]
which, as we observed in
Subsection \ref{Subsec:Ham}, is
the only possible Hamilton
directed path 
that a digraph
in $\A^*(3)$ could
possibly contain.
\section{The Hopf algebra
structure of $\A^*(n)$}
\label{Sec:Hopf}
In this section, we
review the basic theory
of Hopf algebras
and describe the Hopf algebra
structure of $\A^*(n)$.  
We then prove Theorem
\ref{Thm:Hopf} and
Corollary \ref{Cor:Hopf}
and pose some open questions.
Our general reference for the theory
of Hopf algebras is \cite{Sweed}.
\subsection{Hopf algebras}
Let $k$ be a field.  Recall
that a
{\bf Hopf algebra} over $k$ is a
unital associative $k$-algebra
$A=(A,\mu,u)$
equipped with three
additional $k$-linear structure
maps, namely $\Delta:
A\to A\otimes A$ (the coproduct), $\ve:A\to k$ (the counit), and $c:A\to A$ (the
antipode), such that
$\Delta$ is counital and
coassociative, and such that
the following diagram 
commutes:
\[
\xymatrixcolsep{1pc}
\xymatrix{
&A\otimes A\ar[rr]^{c\otimes1}
&&A\otimes A\ar[dr]^{\mu}&\\
A\ar[ur]^{\Delta}\ar[rr]^{\ve}
\ar[rd]_{\Delta}
&&k\ar[rr]^{u}&&A\\
&A\otimes A\ar[rr]_{1\otimes c}
&&A\otimes A\ar[ur]_{\mu}
}
\]
This data amounts to saying that 
a Hopf algebra is a cogroup
object in the category of $k$-algebras,
with $\Delta$ corresponding to the
group operation, $\ve$ corresponding
to the identity, and $c$ corresponding
to inversion.  

Let $A$ be a Hopf algebra 
over $k$.  An ideal $I\subset A$ is
said to be a {\bf Hopf ideal} if 
$\Delta(I)\subset I\otimes A+
A\otimes I$, $\ve(I)=0$, and
$c(I)\subset I$.  If $I\subset A$ is
a Hopf ideal, then the structure maps
of $A$ descend to the quotient
$A/I$, giving $A/I$ the structure
of a Hopf algebra.
% \begin{multicols}{3}
% \xymatrix{
% k\otimes A &
% \ar[l]_{\ve\otimes1}
% A\otimes A
% \ar[r]^{1\otimes\ve}& 
% A\otimes k\\
% & A\ar[ul]^{\cong}
% \ar[u]^{\Delta}
% \ar[ur]_{\cong} &
% }
% \xymatrix{
% A\otimes A\otimes A&
% A\otimes A\ar[l]\\
% A\otimes A\ar[u] & A\ar[u]\ar[l]
% }
% \end{multicols}
\subsection{The dual Steenrod
algebra as a Hopf algebra}
\label{Subsec:Steenrod}
The dual Steenrod algebra
$\A^*=\FF_2[\xi_1,\xi_2,\xi_3,\ldots]$
has the structure of a Hopf algebra
\cite{Milnor},
with coproduct $\Delta$ defined by
\begin{equation}\label{Eqn:Coprod}
\Delta(\xi_i)=\sum_{k=0}^i
\xi_{i-k}^{2^k}\otimes\xi_k,
\end{equation}
counit $\ve$ defined by $\ve(\xi_i)=0$, and
antipode $c$ given recursively by 
\begin{equation}\label{Eqn:Recursive}
\sum_{k=0}^i\xi_{i-k}^{2^k}c(\xi_k)=0.
\end{equation}
These structure maps extend to all
of $\A^*$ by declaring them to be
$\FF_2$-algebra homomorphisms. Milnor
solved the recursion in 
\eqref{Eqn:Recursive} to obtain the
formula
\begin{equation}\label{Eqn:Anti}
c(\xi_i)=\sum_{\pi}
\prod_{k=1}^{\ell(\pi)}
\xi_{\pi(k)}^{2^{\sigma(k)}}
\end{equation}
where the sum is over all
all ordered partitions $\pi$ of
$i$, $\ell(\pi)$ is the length
of $\pi$, $\pi(k)$ is the $k$th part
of $\pi$, and $\sigma(k)$ is the
sum of the first $k-1$ parts of $\pi$.

One can check that the ideals $I(n)\subset\A^*$ defined
in \eqref{Eqn:Ideal}
are Hopf ideals, so that the quotients
$\A^*(n)=\A^*/I(n)$ inherit 
Hopf algebra structures
under the maps $\Delta$,
$\ve$, and $c$ defined in this
Subsection.  

\subsection{Graph theoretic
interpretation of the coproduct
and antipode}
We now prove Theorem
\ref{Thm:Hopf}, starting
with the graph theoretic
interpretation of the coproduct.
We must show that the image of
$\xi_i^{2^j}\in\A^*(n)$ under
$\Delta$
is the sum of tensors of all pairs
of edges that make length 2 directed 
paths from $2^j$ to $2^{i+j}$.  
The
coproduct formula \eqref{Eqn:Coprod} implies
\begin{equation}
\label{Eqn:Coprod}
\Delta(\xi_i^{2^j})=
\Delta(\xi_i)^{2^j}=
\xi_i^{2^j}\otimes1+1\otimes
\xi_i^{2^j}+\sum_{k=1}^{i-1}
\xi_{i-k}^{2^{j+k}}\otimes
\xi_k^{2^j}.
\end{equation}
The first two summands of the
right-hand side of 
\eqref{Eqn:Coprod} represent
degenerate length 2 directed paths
from $2^j$ to $2^{i+j}$.
A non-degenerate
length 2 directed path from
$2^j$ to $2^{i+j}$ corresponds
to a choice of an intermediate
vertex, which in this case would
be of the form $2^{j+k}$ for
some $k$, $1\leq k\leq i-1$.
Given this choice of $k$, the
edge from $2^j$ to $2^{j+k}$ is
$\xi_{j+k-j}^{2^j}=\xi_k^{2^j}$ and the edge from
$2^{j+k}$ to $2^{i+j}$ is 
$\xi_{i+j-(j+k)}^{2^{j+k}}
=\xi_{i-k}^{2^{j+k}}$.  The
terms of the
sum indexed by $k$ 
on the far-right of 
\eqref{Eqn:Coprod} correspond
precisely to the pairs of edges
just described.

Next, we prove the portion of
Theorem \ref{Thm:Hopf} concerning
the graph theoretic interpretation
of the antipode.  
The claim here is that 
the image of $\xi_i^{2^j}$ under
$c$ is the sum of all directed paths
from $2^j$ to $2^{i+j}$.  
Our proof
models that of 
\cite[Lemma 3.1.8]{Yearwood}.  
Milnor's
formula \eqref{Eqn:Anti} for the antipode implies 
\begin{equation}
%\label{Eqn:Anti}
c(\xi_i^{2^j})=c(\xi_i)^{2^j}
=\sum_\pi\prod_{k=1}^{\ell(\pi)}
\xi_{\pi(k)}^{2^{\sigma(k)+j}}
\end{equation}
where $\pi$, $\ell(\pi)$, 
$\pi(k)$, and $\sigma(k)$ are 
defined as in Subsection
\ref{Subsec:Steenrod}.  A directed path from
$2^j$ to $2^{i+j}$ corresponds to a
choice of intermediate vertices, say
\[
2^{j+a_1},2^{j+a_2},\ldots,
2^{j+a_m}
\]
for $0=a_0<a_1<a_2<\cdots<a_m<
a_{m+1}=i$.
The successive differences yield
a unique ordered partition 
$\pi$ of $i$,
namely
\[
i=(a_1-a_0)+(a_2-a_1)+
\cdots+(a_m-a_{m-1})+(i-a_{m})
\]
for which $\ell(\pi)=m+1$,
$\pi(k)=a_k-a_{k-1}$, and
$\sigma(k)=(a_1-0)+(a_2-a_1)+
\cdots+(a_{k-1}-a_{k-2})=a_{k-1}$.
The monomial corresponding to this
directed path is therefore
\begin{align*}
\xi_{j+a_1-j}^{2^{j}}
\xi_{j+a_2-(j+a_1)}^{2^{a_1+j}}
\cdots
\xi_{j+a_m-(j+a_{m-1})}^{2^{a_{m-1}+j}}
\xi_{j+i-(j+a_m)}^{2^{a_m+j}}
&=\xi_{a_1-a_0}^{2^{a_0+j}}
\xi_{a_2-a_1}^{2^{a_1+j}}\cdots
\xi_{a_m-a_{m-1}}^{2^{a_{m-1}+j}}
\xi_{a_{m+1}-a_m}^{2^{a_m+j}}\\
&=\xi_{\pi(1)}^{2^{\sigma(1)+j}}
\xi_{\pi(2)}^{2^{\sigma(2)+j}}
\cdots
\xi_{\pi(m)}^{2^{\sigma(m)+j}}
\xi_{\pi(m+1)}^{2^{\sigma(m+1)+j}}\\
&=\prod_{k=1}^{\ell(\pi)}
\xi_{\pi(k)}^{2^{\sigma(k)+j}}
\end{align*}
which is precisely the general
term of the summation
in \eqref{Eqn:Anti} indexed by
$\pi$.  This completes the proof
of Theorem \ref{Thm:Hopf}.  

Recall from Section \ref{Sec:Intro}
that Theorem \ref{Thm:Hopf} yields
an alternate characterization of 
unilaterality of a directed
graph $x^{\digr}$ underlying
$x\in\A^*(n)$, namely
Corollary \ref{Cor:Hopf}.  This
corollary asserts $x^{\digr}$ is
unilateral if and only if for each
$\xi_i^{2^j}\in\A^*(n)$, at least
one summand of $c(\xi_i^{2^j})$ is a
factor of $x$.
To obtain the corollary,
note that the digraph 
$x^{\digr}$ underlying
$x\in\A^*(n)$
is unilateral if and only if there
is a directed path connecting 
any two of its vertices, say
$2^j$ and $2^{j+i}$.
Theorem \ref{Thm:Hopf} shows
this is equivalent to the demand
that at least one summand of
$c(\xi_i^{2^j})$ appears
as a factor of $x$.
\subsection{Open questions}
We record some outstanding questions related to Wood's
encoding of the algebras
$\A^*(n)$ in terms of graphs.
\begin{enumerate}
\item What is the analog
of Wood's construction for
the mod $p$ dual Steenrod algebra
for odd primes $p$?  What 
characterizations of connectedness, trees, etc., 
are there in these odd primary
situations?
\item In \cite[\S8]{Wood}, Wood
points out that the mod 2 Steenrod
algebra $\A_*$ (as opposed to
its dual, which has been the sole
focus of this paper) and some of
its subalgebras can be interpreted
graph theoretically.  What would 
the results in \cite{Yearwood}
or in this paper look like in
that setting?
\item Given the Hopf algebra
structure on $\A^*(n)$, including
the coproduct $\Delta$ and
antipode $c$, what is the
graph theoretic meaning of 
$\Delta(x)$ and $c(x)$ for an
arbitrary monomial $x\in\A^*(n)$?
\item In \cite[\S5]{Wood}, Wood
describes two procedures one can
perform in the mod 2 Steenrod
algebra $\A_*$, called
{\em stripping} and {\em strapping}, that together 
allow one to
derive all of the Adem relations
from the single relation
$Sq^1Sq^1=0$.  A step in the process of recovering the Adem
relations involves assigning
to each monomial $\xi_i^{2^j}
\in\A^*$ a ``stripping operator''
$\omega$ which is analogous to
how Wood's construction discussed
in this paper
assigns an edge of a graph to
each $\xi_i^{2^j}$. How can 
this analogy be leveraged to obtain further results about
Wood's graph theoretic 
interpretation of $\A^*(n)$?
\item How can Wood's encoding
of $\A^*(n)$ help with
calculations in homotopy theory,
particularly with the Adams
spectral sequence at the prime 2?
\end{enumerate}

\bibliographystyle{plain} % We choose the "plain" reference style
\bibliography{refs}

\end{document}